\documentclass[11pt]{article}
\usepackage{amsfonts}
\usepackage{latexsym}
\usepackage{amsmath}
\usepackage{amssymb}
\usepackage{amsthm}
\usepackage{graphicx}
\usepackage{enumerate}
\usepackage{color}
\usepackage{bbm}
\newtheorem{theorem}{Theorem}[section]
\newtheorem{lemma}[theorem]{Lemma}

\newtheorem{proposition}[theorem]{Proposition}

\newtheorem{problem}[theorem]{Problem}
\theoremstyle{definition}

\newtheorem{example}[theorem]{Example}

\newtheorem{thmy}{Theorem}
\newenvironment{oldtheorem}{\stepcounter{thm}\begin{thmy}}{\end{thmy}}
\newtheorem*{note*}{Note}

\makeatletter

\newcommand{\R}{\mathbb R}
\newcommand{\Rn}{{\mathbb R}^n}

\newcommand{\Sn}{\mathbb{ S}^{n-1}}
\newcommand{\Snm}{{\mathbb S}^{n-2}}
\newcommand{\Ha}{{{\cal H}^{n-1}}}
\newcommand{\Ham}{{{\cal H}^{n-2}}}
\newcommand{\Bo}{{\cal B}}
\makeatother

\begin{document}


\title{\bf Functions with isotropic sections}
\date{}
\medskip

\author {Ioannis Purnaras, Christos Saroglou}

\maketitle
\begin{abstract}
We prove a local version of a recently established theorem by Myroshnychenko, Ryabogin and the second named author. More specifically, we show that if $n\geq 3$, $g:\mathbb{S}^{n-1}\to\mathbb{R}$ is an even bounded measurable function, $U$ is an open subset of $\mathbb{S}^{n-1}$ and the restriction (section) of $f$ onto any great sphere perpendicular to $U$ is isotropic, then ${\cal C}(g)|_U=c+\langle a,\cdot\rangle$ and ${\cal R}(g)|_U=c'$, for some fixed constants $c,c'\in\mathbb{R}$ and for some fixed vector $a\in \mathbb{R}^n$. Here, ${\cal C}(g)$ denotes the cosine transform and ${\cal R}(g)$ denotes the Funk transform of $g$. However, we show that $g$ does not need to be equal to a constant almost everywhere in $U^\perp:=\bigcup_{u\in U}(\mathbb{S}^{n-1}\cap u^\perp)$.  For the needs of our proofs, we obtain a new generalization of a result from classical differential geometry, in the setting of convex hypersurfaces, that we believe is of independent interest. 
\end{abstract}
\section{Introduction}

\hspace*{1.5em}Let us fix an orthonormal basis $\{e_1,\dots,e_n\}$ in $\mathbb{R}^n$. We write $\langle x, y\rangle$ for the standard inner product of $x$ and $y$ in $\R^n$. For $k=1,\dots,n-1$, the set of all $k$-dimensional subspaces of $\Rn$ is denoted by $G_{n,k}$. If $A\subseteq \Rn$, the orthogonal projection of $A$ onto a subspace $H\in G_{n,k}$, will be denoted by $A|H$. If $u\in\Rn$, we denote by $u^\perp$ the subspace of codimension 1 which is orthogonal to $u$. The notation $B_2^n$ stands for the standard unit ball in $\mathbb{R}^n$. Also, $\Sn=\{x\in \Rn:|x|=1\}$ denotes the
unit sphere in $\R^n$. The boundary of a set $A$ will be denoted by $\textnormal{bd}A$. A {\it spherical cap} $U\subseteq \Sn$ is any set of the form $\{x\in \Sn:\langle x,u\rangle>a\}$, $0<a<1$, $u\in \Sn$. The point $u$ is called the {\it center} of the spherical cap $U$. 
Denote, also, by ${\cal H}^a$, the $a$-dimensional Hausdorff measure in $\Rn$, where $0<a\leq n$. We will say that a Borel measure on the sphere $\Sn$ is absolutely continuous if it is absolutely continuous with respect to $\Ha$. For a Borel set $\omega$ in $\Sn$, $\Bo(\omega)$ stands for the $\sigma$-algebra of Borel subsets of $\omega$.  Any convergence of sets will be with respect to the Hausdorff metric. The orthogonal group in $\Rn$ is denoted by $O(n)$. For $u\in\Sn$, we set $O(n,u):=\{T\in O_n:Tu=u\}$. 

Let $\mu$ be a signed Borel measure on $\Sn$ and $\zeta:\Sn\to\R$ be an integrable function. The {\it cosine transform} ${\cal C}(\mu)$ of $\mu$ and the {\it Funk transform} (=Radon transform on the sphere) ${\cal R}(\zeta)$ of $\zeta$ are defined as follows.
$${\cal R}(\zeta)(u)=\int_{\Sn\cap u^\perp}\zeta (x)d\Ham(x),\qquad u\in \Sn,$$ 
$${\cal C}(\mu)(u)=\int_{\Sn}|\langle x,u\rangle| d\mu(x),\qquad u\in \Sn.$$
If $f$ is an integrable function on $\Sn$, we define ${\cal C}(f):={\cal C}(fd\Ha(\cdot))$ and we simply say that ${\cal C}(f)$ is the cosine transform of $f$.
A function $g:\Sn\to\R$ is called {\it isotropic} if the map 
$$\Sn\ni u\mapsto \int_{\Sn}\langle x,u\rangle^2 g(x)d\Ha(x)$$is constant. The following problem was proposed in \cite{M-R-S}.
\begin{problem}\label{prob1}
Assume that for a measurable subset $U$ of $S^{n-1}$ and for an even bounded measurable function $g:S^{n-1}\to\mathbb{R}$, the restriction $g|_{S^{n-1}\cap u^{\perp}}$ onto $S^{n-1}\cap u^{\perp}$ is isotropic, for almost all $u\in U$. Is it true that $g$ is almost everywhere equal to a constant on the set $U^\perp$?
\end{problem}
Here, $U^\perp$ stands for the union of all great subspheres of $\Sn$, which are orthogonal to a direction from $U$, i.e $U^\perp=\bigcup_{u\in U} (S^{n-1}\cap u^{\perp})$. The following was established in \cite{M-R-S}. 
\begin{oldtheorem}\label{thm-old-s}
Problem \ref{prob1} has affirmative answer if $U=\Sn$. 
\end{oldtheorem}
Our goal is to prove that the answer to Problem \ref{prob1} is in general negative but on the other hand, a local version of Theorem \ref{thm-old-s} is still valid. 
\begin{theorem}\label{thm-counterexample}
Let $U$ be an open subset of $\Sn$, that does not contain $U^\perp$. There exists a continuous function $g:\Sn\to \R$, such that for any $u\in U$, $g|_{\Sn\cap u^\perp}$ is isotropic, but $g$ is not constant on $U^\perp$.
\end{theorem}
\begin{theorem}\label{thm-meta-main-1}
Let $n\geq 3$, $U$ be an open subset of $\Sn$ and $g:U\to\R$ be an even, bounded, measurable function. If for almost every $u\in U$, $g|_{\Sn\cap u^\perp}$ is isotropic, then ${\cal C}(g)|_U=c+\langle a,\cdot\rangle$ and ${\cal R}(g)=c'$, almost everywhere in $U$, for some fixed constants $c,c'\in\mathbb{R}$ and for some fixed vector $a\in \mathbb{R}^n$.
\end{theorem}
The fact that Theorem \ref{thm-meta-main-1} is a local version of Theorem \ref{thm-old-s} follows from the classical fact that if ${\cal C}(g)$ is constant on $\Sn$, then $g$ has to be almost everywhere equal to a constant on $\Sn$. In fact, the proof of Theorem \ref{thm-old-s}, is based on Theorem \ref{thm-meta-main-1}, proved in \cite{M-R-S} in the case $U=\Sn$. The proof of the latter relies on a quick ``global" argument based on the Aleksandrov-Fenchel inequality (see next section). However, such arguments will not work in the local setting.

For a strictly convex body $K$ with $C^2$ smooth boundary and a direction $u\in \Sn$, denote by $r_K^1(u),\dots,r_K^{n-1}(u)$ the principal radii of curvature of $K$ at $u$ (see next section). It is well known that 
\begin{equation}\label{eq:rad-curv}
r_K^i(u)=\frac{1}{k_K^i(v_K(u))}, \qquad i=1,\dots,n-1,
\end{equation}
where $k_K^1(x),\dots,k_K^{n-1}(x)$ are the principal curvatures of the hypersurface $\textnormal{bd}K$ at the point $x\in\textnormal{bd}K$. Here, $v_K:\Sn\to\textnormal{bd}K$ denotes the inverse Gauss map , i.e. for $u\in\Sn$, $v_K(u)$ is the (unique) point of intersection of $K$ with its supporting hyperplane whose outer unit normal vector is $u$.

The proof of the general case of Theorem \ref{thm-meta-main-1} exploits the following observation that we believe is new: If $g$ is smooth enough and $g|_{\Sn\cap u^\perp}$ is isotropic for some $u\in\Sn$, then the principal curvatures of the boundary of the zonoid $Z(g)$, whose generating measure is given by $S_{n-1}(Z(g),\cdot)=gd\Ha(\cdot)$ (see Section 4), at $v_K(u)$ are all equal. That is, the point $v_k(u)$ is an umbilic of the boundary of $Z(g)$. Therefore, if $g$ is smooth enough, one can use the following classical result (see e.g. \cite[pp 183]{Doc}) to prove Theorem \ref{thm-meta-main-1}.

\begin{oldtheorem}\label{thm-old}
Let $V$ be a hypersurface in $\Rn$, $n\geq 3$, of class $C^3$ (or according to \cite{S-V}, of class $C^2$). If for all $x\in V$, it holds $0\neq k_1(x)=\dots=k_{n-1}(x)\in\mathbb{R}$, then $V$ is contained in a Euclidean sphere, where $k_1(x),\dots,k_{n-1}(x)$ are the principal curvatures of $V$ at $x$.
\end{oldtheorem}
The reader might guess that, since we do not assume any regularity on $g$, Theorem \ref{thm-old} cannot be used directly (to our knowledge, not even if we assume $g$ to be continuous) to prove Theorem \ref{thm-meta-main-1}. Thus, we need somehow to relax the regularity assumptions in Theorem \ref{thm-old}, at least in the convex case. This is done in the following theorem, which we believe is of independent interest.

\begin{theorem}\label{thm-main} Let $K$ be a convex body in $\mathbb{R}^n$, $n\geq 3$, $U$ be an open connected subset of $\mathbb{S}^{n-1}$ and assume that the measure $S_1(K,\cdot)|_{\Bo(U)}$ is absolutely continuous. If for almost every direction $u\in U$ it holds
\begin{equation}\label{eq:equal-princ-radii}
r_K^1(u)=\dots=r_K^{n-1}(u),
\end{equation}
then $\tau(K,U)$ is contained in a Euclidean sphere.
\end{theorem}
Here, $S_1(K,\cdot)|_{\Bo(U)}$ denotes the order 1 area measure of $K$, restricted into the family of Borel subsets of $U$ and $\tau(K,U)$ is the inverse spherical image of $U$ with respect to $K$. We refer to the next section for definitions.


Theorem \ref{thm-main} is in some sense optimal. This is demonstrated in the following examples.

\begin{example}\label{ex-1}
One cannot replace (\ref{eq:equal-princ-radii}) by the condition that for almost every point in an open subset of $\textnormal{bd}K$, the principal curvatures are equal. To see this, take $K$ to be the intersection of two Euclidean balls with different centers. 
\end{example}

\begin{example}\label{ex-2}
The assumption that $S_1(K,\cdot)|_{\Bo(U)}$ is absolutely continuous cannot be removed. Indeed, take for instance $K$ to be the Minkowski sum of a Euclidean ball and a polytope and $U=\Sn$.
\end{example}
Nevertheless, we do not know whether the assumption of absolute continuity of the order 1 area measure (restricted in ${\cal B}(U)$) in Theorem \ref{thm-main} can be replaced by the absolute continuity of the area measure of any other order.

The main tools for the proof of our results come from Convex and Integral Geometry. This paper is structured as follows. In Section 2, we provide the necessary background for the proof of our main results. Theorem \ref{thm-main} is proved in Section 3. In Section 4, we prove Theorems \ref{thm-counterexample} and \ref{thm-meta-main-1} and, under some regularity assumptions on $g$, a local version of Theorem \ref{thm-meta-main-1}.
\section{Preliminaries and notation}
\hspace*{1.5em}In this section we introduce notation and collect basic facts from classical theory of convex bodies that we use in the paper. As a general reference on the theory we use R. Schneider's book ``Convex bodies: the Brunn-Minkowski theory" \cite{Sc} (see also \cite{B-F} or \cite{G}).

Let $A,\ B$ be subsets of $\R^n$. The {\it linear hull} of $A$ is denoted by $\textnormal{span} A$. The {\it Minkowski sum} $A+B$ of $A$ and $B$ is the set $\{x+y:x\in A,y\in B\}$ .

A {\it convex body} $K$ in $\Rn$ is a convex compact set with non-empty interior.
The function $h_K:\Rn\to\R$, with 
$h_K(u)=\max\{\langle x, u\rangle : x \in K\}$ is the {\it support function} of $K$. The support functional is known to be additive with respect to the Minkowski sum and 1-homogeneous. That is, $h_{\lambda K+\mu L}=\lambda h_K+\mu h_L$, for any compact convex sets $K,L$ and for any $\lambda,\mu\geq 0$. Moreover if $H$ is a subspace of $\Rn$ and $T:\Rn\to\Rn$ is any orthogonal map, then the following identities hold:
$$h_{K|H}=(h_K)|_H\qquad\textnormal{and}\qquad h_{TK}=h_K\circ T^{\ast},$$where $T^\ast$ denotes the adjoint of $T$.  

For a convex body $K$ and $u\in\Sn$, the support set $F(K,u)$ of $K$ in the direction $u$ is defined by $F(K,u)=\{x\in K:\langle x,u\rangle=h_K(u)\}$. Similarly with the support functional, the support set functional is additive with respect to the Minkowski sum. That is, if $L$ is another convex body, then
\begin{equation}\label{eq:support-set}
F(K+L,u)=F(K,u)+F(L,u).
\end{equation}

A classical theorem of Minkowski says that if $K_1, K_2, \dots, K_n$ are convex compact sets in $\R^n$ and $\lambda_1, \dots, \lambda_n \ge 0,$ then the volume of the set
$\lambda_1 K_1 +\lambda_2 K_2+\dots + \lambda_nK_n$ is a homogeneous polynomial in $\lambda_1,\dots, \lambda_n$ 
of degree $n$, with non-negative coefficients.
The coefficient of $\lambda_1\cdots \lambda_n$ is called the {\it mixed volume }of $K_{1}, \dots, K_{n}$ and is denoted by
$V(K_{1}, \dots, K_{n})$. We will also write $V(K_1[m_1],\dots, K_r[m_r])$ for the
mixed volume of  $K_1,\dots, K_r$ where each $K_i$ is repeated $m_i$ times and
$m_1+\dots+m_r=n$. 

The Aleksandrov--Fenchel inequality states the following
\begin{equation}\label{eq:af}
V(K_1,K_2, K_3, \dots, K_n)^2 \ge V(K_1,K_1, K_3, \dots, K_n)V(K_2,K_2, K_3, \dots, K_n).
\end{equation}

It turns out that for given convex bodies $K_1,\dots,K_{n-1}$, there is a unique Borel measure $S(K_1,\dots,K_{n-1},\cdot)$ on the sphere $\Sn$, such that for any convex body $L$, it holds
\begin{equation}\label{eq:mixarea}
V(L,K_1,\dots, K_{n-1})=\frac{1}{n}\int_{\Sn}h_L(u)dS(K_1,\dots, K_{n-1},u).
\end{equation}
Similarly, as with mixed volumes, the notation $S(K_1[m_1],\dots, K_r[m_r],\cdot)$ means that $K_i$ is repeated $m_i$ times, $i=1\dots,r$, where $m_1+\dots+m_r=n-1$.
One of the fundamental properties of mixed area measures is additivity and homogeneity with respect to any of its arguments. That is,
\begin{eqnarray}\label{eq:mixedarea-additivity}
&&S(K_1,\dots,K_{m-1},\lambda K_m+\mu K'_m,K_{m+1},\dots,K_{n-1},\cdot)\nonumber\\
&=&\lambda S(K_1,\dots,K_{m-1},K_m,K_{m+1},\dots,K_{n-1},\cdot)+\mu S(K_1,\dots,K_{m-1},K'_m,K_{m+1},\dots,K_{n-1},\cdot),
\end{eqnarray} 
for any convex body $K'_m$ and any numbers $\lambda,\mu>0$.

A useful fact concerning mixed area measure is that if $\{L^{(m)}_{j}\}_{m=1}^\infty$ is a sequence of convex bodies, converging to $K_j$, in the Hausdorff metric, where $j=1,\dots,n-1$, then the corresponding sequence $ \{S(L^{(m)}_1,\dots,L^{(m)}_{n-1},\cdot)\}_{m=1}^\infty$ of mixed area measures converges weakly to $S(K_1,\dots,K_{n-1},\cdot)$. That is, for every continuous function $\varphi:\Sn\to\R$, it holds
$$\int_{\Sn}\varphi dS(L^{(m)}_1,\dots,L^{(m)}_{n-1},\cdot)\xrightarrow{m\to\infty}\int_{\Sn}\varphi dS(K_1,\dots,K_{n-1},\cdot).$$

Let $u\in S^{n-1}$ be a point at which $h_K$ is twice differentiable. If $\{\varepsilon_1,\dots,\varepsilon_{n-1}\}$ is an orthonormal basis of $u^\perp$, we denote by $Hess (h_{K})(u)$ the $(n-1)\times(n-1)$ Hessian matrix of the restriction of $h_K$ onto $T_u \Sn$ (the tangent hyperplane of $\Sn$ at $u$), where we differentiate with respect to the basis $\{\varepsilon_1,\dots,\varepsilon_{n-1}\}$. The eigenvalues $r_K^1(u),\dots,r_K^{n-1}(u)$ of this matrix 
are non-negative (since $h_K$ is convex), independent of the choice of the orthonormal basis $\{\varepsilon_1,\dots,\varepsilon_{n-1}\}$ of $u^\perp$ and are called  ``the principal radii of curvature" of $K$ at $u$. 

We say that a convex body $K$ is of class ${\cal C}^2_+$ if $h_{K}$ is of class $C^2$ and if all the principal radii of curvature of $K$ at any $u\in \Sn$ are strictly positive. If the convex bodies $K_1,\dots,K_{n-1}$ are of class ${\cal C}^2_+$, then the mixed area measure $S(K_1,\dots,K_{n-1},\cdot)$ is absolutely continuous and its density depends pointwise only on the Hessian matrices $Hess (h_{K_i})(u)$, $i=1,\dots,n-1$ but not on the (common) choice of the orthonormal basis $\{\varepsilon_1,\dots,\varepsilon_{n-1}\}$. In fact,
\begin{equation}\label{eq:mixed-discriminant}
\frac{dS(K_1,\dots,K_{n-1},\cdot)}{d\Ha(\cdot)}(u)=D(Hess(h_{K_1})(u),\dots,Hess(h_{K_{n-1}})(u)),
\end{equation} 
where the last expression is the mixed discriminant of the matrices  $Hess(h_{K_1})(u),\dots,Hess(h_{K_{n-1}})(u)$ (see \cite[Section 2.5]{Sc} and the references therein).

If $\omega$ is a subset of $\Sn$, define the {\it inverse spherical image} $\tau(K,\omega)$ of $\omega$ with respect to $K$ by
$$\tau(K,\omega)=\big\{x\in\partial K:\exists u\in\omega,\textnormal{ such that }\langle x,u\rangle=h_K(u)\big\}.$$
Assume, furthermore that $K$ is of class ${\cal C}_+^2$. Since the inverse Gauss map $v_K:\Sn\to  \textnormal{bd}K$ is well defined and continuous, and since in this case it clearly holds $\tau(K,\omega)=v_K^{-1}(\omega)$, it follows that if $\omega$ is an open set in $\Sn$ then $\tau(K,\omega)$ is also open in $\textnormal{bd}K$.

For $j=1,\dots,n-1$, the {\it area measure of order $j$} of a convex body $K$ is defined as
$$S_j(K,\cdot):=S(K[j],B_2^n[n-1-j],\cdot).$$
In particular (as it follows from (\ref{eq:mixedarea-additivity})), the order 1 area measure is additive and homogeneous, i.e. $S_1(\lambda K+\mu L,\cdot)=\lambda S_1(K,\cdot)+\mu S_1(L,\cdot)$, for any $\lambda,\mu>0$ and any convex bodies $K,\ L$.

The special case $j=n-1$ in the previous definition is better understood and of particular interest. The area measure $S_{n-1}(K,\cdot)$ is called the {\it surface area measure} of $K$. The following formula is valid
\begin{equation}\label{eq:surface-area-measure}
S_{n-1}(K,\omega)={\cal{H}}^{n-1}\big(\tau(K,\omega)\big),
\end{equation}
for any Borel  $\omega \subset \Sn$. In addition, Minkowski's Existence and Uniqueness theorem states that any Borel measure, whose center of mass is at the origin and is not concentrated in any great subsphere of $\Sn$, is the surface area measure of a unique (up to translation) convex body.

The density of the absolutely continuous part (in its Lebesgue decomposition) of $S_j(K,\cdot)$ will be denoted by $f_K^{(j)}$. 
Densities of area measures behave well under the action of orthogonal maps. If $T\in O(n)$, then (see \cite{Lut}) \begin{equation}\label{eq:density-orth-maps}
f^{(j)}_{TK}=f^{(j)}_K\circ T^\ast.
\end{equation}

Recall the definition of the elementary symmetric functions $s_j$: If $a_1,\dots,a_{n-1}$ are positive reals, then $$s_j(a_1,\dots,a_{n-1}):={\displaystyle{\binom{n-1}{j}}}^{-1}\displaystyle{\sum_{1\leq i_1<\dots<i_j\leq n-1}a_{i_1}\dots a_{i_j}}.$$
The classical Newton inequality states that if $1\leq i<j\leq n-1$
\begin{equation}\label{eq:Newton}
s_i(a_1,\dots,a_{n-1})^{1/i}\geq s_j(a_1,\dots,a_{n-1})^{1/j}, 
\end{equation}
with equality if and only if $a_1=\dots=a_{n-1}$.

Recall that the support function $h_K$ of the convex body $K$ is twice differentiable for almost every $u\in \Sn$. It is known (see \cite{Hug1}, \cite{Hug2}, \cite{Hug3} for additional information, references and related results concerning area measures and their densities) that $f^{(j)}_K$ is given by
\begin{equation}\label{eq:area-measure-density}
f^{(j)}_K(u)=s_j(r_K^1(u),\dots,r_K^{n-1}(u)),\qquad\textnormal{for almost every }u\in \Sn.
\end{equation}
In the case $j=1$, we can rewrite (\ref{eq:area-measure-density}) as follows
\begin{equation}\label{eq:distributions}
f^{(1)}_K(u)=\frac{1}{n-1}\Delta_S h_K(u)+h_K(u),\qquad\textnormal{for almost every }u\in \Sn,
\end{equation} 
where $\Delta_S$ is the Laplacian (i.e. the Laplace-Beltrami operator) on the sphere. It is well known that the support function of a convex body, restricted on $\Sn$ is contained in the Sobolev space $\mathbb{H}^1(\Sn)$ (see \cite{Kid}, where higher regularity is established). Moreover, as shown in \cite{Be}, (\ref{eq:distributions}) actually holds in the sense of distributions.

We have the following simple Lemmas.
\begin{lemma}\label{l-1-Newton}
Let $K$ be a convex body in $\mathbb{R}^n$, $n\geq 3$, $\omega$ be a Borel subset of $\mathbb{S}^{n-1}$ and $1\leq i\leq j<n-1$. The following statements are equivalent.
\begin{enumerate}[i)]
\item $\left(f^{(i)}_{K}(u)\right)^{1/i}=\left(f^{(j)}_{K}(u)\right)^{1/j}$, for almost every $u\in \omega$.
\item $\left(f^{(i)}_{K}(u)\right)^{1/i}\leq\left(f^{(j)}_{K}(u)\right)^{1/j}$, for almost every $u\in \omega$.
\item $r^1_K(u)=\dots=r^{n-1}_K(u)$, for almost every $u\in \omega$. 
\end{enumerate} 
\end{lemma}
\begin{proof} Using Newton's inequality (\ref{eq:Newton}) together with the representation (\ref{eq:area-measure-density}) of the densities $f^{(i)}_{K}$, $f^{(j)}_{K}$, we obtain
$$\left(f^{(i)}_{K}(u)\right)^{1/i}= s_i\left(r_K^1(u),\dots,r_K^{n-1}(u)\right)^{1/i}\geq s_j\left(r_K^1(u),\dots,r_K^{n-1}(u)\right)^{1/j}=\left(f^{(j)}_{K}(u)\right)^{1/j},$$
for almost every $u\in \omega$. Therefore, if $(i)$ or $(ii)$ holds, then we have equality in Newton's inequality (\ref{eq:Newton}), which is only possible if $r^1_K(u)=\dots=r^{n-1}_K(u)$, for almost every $u\in \omega$. Conversely, if $(iii)$ holds, then by (\ref{eq:area-measure-density}), $(i)$ and $(ii)$ trivially hold true.
\end{proof}
\begin{lemma}\label{l-2-Newton}
Let $K_1$, $K_2$ be convex bodies in $\Rn$, satisfying the assumptions of Theorem \ref{thm-main} for some open set $U$ in $\Sn$. Then, for $\lambda>0$, the convex body $\lambda(K_1+K_2)$ also satisfies the assumptions of Theorem \ref{thm-main} for $U$.
\end{lemma}
\begin{proof}
Notice, first, that by the additivity and homogeneity of the order 1 area measure, we have $S_1(\lambda(K_1+K_2),\cdot)=\lambda S_1(K_1,\cdot)+\lambda S_2(K_2,\cdot)$. Hence, $S_1(\lambda(K_1+K_2),\cdot)|_{\Bo (U)}$ is absolutely continuous. Moreover, it holds $r^1_{K_i}(u)=\dots=r^{n-1}_{K_i}(u)$, $i=1,2$, for almost every $u\in U$. Thus, $Hess(h_{K_i})(u)=r^1_{K_i}(u)I_{(n-1)\times (n-1)}$, for almost every $u\in U$, where $I_{(n-1)\times(n-1)}$ stands for the $(n-1)\times (n-1)$ identity matrix. This, together with the additivity and homogeneity of the support functional, gives
\begin{eqnarray*}Hess(h_{\lambda(K_1+K_2)}(u))&=&Hess(\lambda h_{K_1}+h_{K_2})(u)=\lambda\left(Hess(h_{K_1})(u)+Hess(h_{K_2}(u)\right)\\
&=&\lambda (r^1_{K_1}(u)+r^1_{K_2}(u))I_{(n-1)\times(n-1)},
\end{eqnarray*}
for almost every $u\in U$, proving our claim.
\end{proof}
We will also need two statements from basic measure theory (which of course hold in a much more general setting).
\begin{lemma}\label{l-mt}
Let $\mu,\nu_1,\nu_2,\xi$ be Borel measures on an open set $U$ in $\Sn$. 
\begin{enumerate}[i)]
\item If $\int_U\varphi d\nu_1\leq \int_U\varphi d\nu_2$, for all continuous non-negative functions $\varphi$ supported on $U$, then $\nu_1\leq \nu_2$.
\item If $\nu_i=f_id\mu$ (i.e. $\nu_i$ is absolutely continuous with density $f_i$ with respect to $\mu$), $i=1,2$ and $\mu$, $\xi$ are mutually singular measures and $\nu_1\leq \nu_2+\xi$, then $f_1\leq f_2$, $\mu$-almost everywhere.
\end{enumerate}
\end{lemma}
\begin{proof}
We only prove
$(ii)$, since $(i)$ is well known. Clearly, for $\varepsilon>0$, there exists a Borel set $A_\varepsilon\subseteq U$, such that $\mu(U\setminus A_\varepsilon)<1/\varepsilon$ and $\xi (A_\varepsilon)=0$. Then, for any Borel subset $B$ of $A_\varepsilon$, we have $\int_Bf_1d\mu=\nu_1(B)\leq \nu_2(B)=\int_Bf_2d\mu$. It follows that $f_1|_{A_\varepsilon}\leq f_2|_{A_\varepsilon}$, $\mu$-almost everywhere. Thus, $\mu(\{f_1>f_2\})<1/\varepsilon$ and, since $\varepsilon$ is arbitrary, our assertion follows. 
\end{proof}
\section{Convex umbilical hypersurfaces}
\hspace*{1.5em}For the proof of Theorem \ref{thm-main}, we will show that if some pair $(K,U)$ satisfies the assumptions of the theorem, then $h_K$ is smooth enough. Theorem \ref{thm-main} will then follow from Theorem \ref{thm-old}. To this end, we will show that $f^{(1)}_K$ actually has to be harmonic on $U$, which by general theory of elliptic PDE's, will give us the desired regularity of $h_K$.

\subsection{Symmetrization}
\hspace*{1.5em}Let $f:\Sn\to\R$ be a non-negative measurable function. The \emph{radial symmetrization} $Sr(f)$ of $f$ with respect to the line $\R e_n$ is defined as follows.
\begin{equation}\label{eq-Sr-def}
Sr(f)(u):=\frac{\int_{\{x_n=u_n\}}f(x)d\Ham(x)}{\Ham(\{x_n=u_n\})}.
\end{equation} 
The operator $S_r(\cdot)$ corresponds to the so-called ``Blaschke-Minkowski" symmetrization, when applied to the support function of a convex body. We refer to \cite{Bi-Ga-Gr} and \cite{B-F} for more information. In view of Lemma \ref{l-2-Newton}, one naturally expects that there is some sequence of averages of compositions of $f$ with maps from $O(n,e_n)$ that converges in some sense to $Sr(f)$. Since we are going to need convergence in $L^2$, we will do this process carefully. 

It is clear that $Sr(f)$ is invariant under composition with maps from $O(n,e_n)$. Moreover, $Sr(g)=g$, for any function $g$ that is radially symmetric with respect to the line $\R e_n$; that is, $Sr$ is an idempotent operator. Furthermore, an immediate application of H\"{o}lder's inequality yields
\begin{equation}\label{eq-Sr-CS}
Sr(f)(u)\leq (Sr(f^p)(u))^{1/p},\qquad p\geq 1,\qquad u\in \Sn.
\end{equation}
Later on, we will need the fact that the $L^1$-norm is preserved under the operator $Sr(\cdot)$ (this is mentioned in \cite{Bi-Ga-Gr}) and that if $f$ is in $L^2$, then $Sr(f)$ is also in $L^2$. This is done in the following lemma.
\begin{lemma}\label{l-sr-l2}
Let $f:\Sn\to\R$ be a non-negative measurable function. Then, for any $v\in \Sn\cap e_n^\perp$, it holds
\begin{equation}\label{eq-sr-l2}
\|f\|_{L^1(\Sn)}=(n+1)(n-1)\omega_{n-1}\int_{-1}^1\int_0^{\sqrt{1-t^2}}r^{n-2}\sqrt{r^2+t^2}Sr(f)\left(v+\frac{t}{\sqrt{r^2+t^2}}e_n\right)drdt,
\end{equation}
where $\omega_n$ is the volume of $B_2^n$. In particular, we have $\|f\|_{L^1(\Sn)}=\|Sr(f)\|_{L^1(\Sn)}$ and, for $p>1$, $\|f\|_{L^p(\Sn)}\geq\|Sr(f)\|_{L^p(\Sn)}$.
\end{lemma}
\begin{proof} Fix $v\in \Sn\cap e_n \equiv \Snm$ and let $r>0$, $t\in\R$, $\gamma\in \Snm$. Since $\langle(r\gamma,t)/|(r\gamma,t)|,e_n\rangle=t/\sqrt{r^2+t^2}$, an easy change of variables implies
\begin{eqnarray}
\frac{1}{\Ham(\Snm)}\int_{\Snm}f\left(\frac{(r\gamma,t)}{|(r\gamma,t)|}\right)d\Ham(\gamma)&=&\frac{\int_{\left\{x_n=t/\sqrt{r^2+t^2}\right\}}f(x)d\Ham(x)}{\Ham(\{x_n=\sqrt{r^2+t^2}\})}\nonumber\\
&=&Sr(f)\left(v+\frac{t}{\sqrt{r^2+t^2}}e_n\right).\label{eq-l-sr-l2-1}
\end{eqnarray}
Extend $f$ to the whole $\Rn$, so that $f:\Rn\to\R$ is 1-homogeneous. Integrating in polar coordinates, we obtain
\begin{eqnarray*}
\int_{B_2^n}f(x)dx=\int_{\Sn}\int_0^1f(r\gamma)r^{n-1}drd\Ha(\gamma)=\int_{\Sn}f(\gamma)d\Ha(\gamma)\int_0^1r^ndr=\frac{1}{n+1}\int_{\Sn}f(\gamma)d\Ha(\gamma).
\end{eqnarray*}
Therefore, using Fubini's theorem, (\ref{eq-l-sr-l2-1}) and again integration in polar coordinates, we get
\begin{eqnarray}
\|f\|_{L^1(\Sn)}&=&(n+1)\int_{B_2^n}f(x)dx=(n+1)\int_{-1}^1\int_{B_2^n\cap(e_n^\perp+te_n)}f(y,t)dydt\nonumber\\
&=&(n+1)\int_{-1}^1\int_0^{\sqrt{1-t^2}}\int_{\Snm}f(r\gamma,t)d\Ham(\gamma)r^{n-2}drdt\label{eq-l-sr-l2-2}\\
&=&(n+1)\int_{-1}^1\int_0^{\sqrt{1-t^2}}\sqrt{r^2+t^2}r^{n-2}\int_{\Snm}f\left(\frac{(r\gamma,t)}{|(r\gamma,t)|}\right)d\Ham(\gamma)drdt\nonumber\\
&=&(n+1)(n-1)\omega_{n-1}\int_{-1}^1\int_0^{\sqrt{1-t^2}}r^{n-2}\sqrt{r^2+t^2}Sr(f)\left(v+\frac{t}{\sqrt{r^2+t^2}}e_n\right)drdt,\nonumber
\end{eqnarray}
as required. The fact that $\|f\|_{L^1(\Sn)}=\|Sr(f)\|_{L^1(\Sn)}$ follows immediately from (\ref{eq-sr-l2}) and the fact that $Sr$ is idempotent. Similarly, using (\ref{eq-Sr-CS}), we get
\begin{eqnarray*}
\|Sr(f)\|^p_{L^p(\Sn)}&=&(n+1)(n-1)\omega_{n-1}\int_{-1}^1\int_0^{\sqrt{1-t^2}}r^{n-2}\sqrt{r^2+t^2}Sr(Sr(f)^p)\left(v+\frac{t}{\sqrt{r^2+t^2}}e_n\right)drdt\\
&=&(n+1)(n-1)\omega_{n-1}\int_{-1}^1\int_0^{\sqrt{1-t^2}}r^{n-2}\sqrt{r^2+t^2}Sr(f)^p\left(v+\frac{t}{\sqrt{r^2+t^2}}e_n\right)drdt\\
&\leq &(n+1)(n-1)\omega_{n-1}\int_{-1}^1\int_0^{\sqrt{1-t^2}}r^{n-2}\sqrt{r^2+t^2}Sr(f^p)\left(v+\frac{t}{\sqrt{r^2+t^2}}e_n\right)drdt\\
&=&\|f\|^p_{L^p(\Sn)}.
\end{eqnarray*}
\end{proof}
Let $f:\Sn\to\R$. For $T_1,\dots,T_m\in O(n,e_n)$, define the function
$$M(f;T_1,\dots,T_m):=\frac{f\circ T_1+\dots+f\circ T_m}{m}.$$

\begin{proposition}\label{prop-conv-sr}
Let $f_1,\dots,f_k:\Sn\to\R$ be $L^2$-functions. Then, there exists a sequence $T_1^1,\dots ,T^1_{m_1},T_1^2,\dots, T_{m_2}^2,\dots \in O(n,e_n)$, such that $$M(f_i;T_1^j,\dots ,T_{m_j}^j)\xrightarrow{j\to\infty}Sr(f_i),\qquad i=1,\dots,k,$$
in $L^2(\Sn)$.
\end{proposition}
\begin{proof}
Consider the linear space $X:=(L^2(\Sn))^k$ equipped with the natural norm given by $\|(w_1,\dots,w_k)\|^2=\sum_{i=1}^k\|w_i\|_{L^2(\Sn)}$. Then, the pair $(X,\|\cdot\|)$ is a Hilbert space.
Define the set $${\cal A}:=\{(M(f_1;T_1\dots,T_m),\dots,M(f_k;T_1\dots,T_m)):m\in \mathbb{N}, \ T_1,\dots,T_m\in O(n,e_n)\}$$
and observe that the closure ${\cal C}:=cl {\cal A}$ (with respect to the norm $\|\cdot\|$) of ${\cal A}$ is a convex set. To see this, notice that since ${\cal A}$ is clearly closed under rational convex combinations, its closure has to be closed under (any) convex combinations. Using a classical result from the theory of Hilbert spaces (see e.g. \cite[Chapter 3]{D}), we conclude that there exists a unique element $(g_1,\dots,g_k)\in{\cal C}$, such that 
$$\|(g_1,\dots,g_k)-(Sr(f_1),\dots,Sr(f_k))\|=\inf\left\{\|(w_1,\dots,w_k)-(Sr(f_1),\dots,Sr(f_k))\|:(w_1,\dots,w_k)\in{\cal C}\right\}=:d.$$ 
It suffices to prove that $g_i=Sr(f_i)$ almost everywhere in $\Sn$. Indeed, then there will be a sequence from ${\cal C}$ that converges to $(Sr(f_1),\dots,Sr(f_k))$ in $L^2$.
Observe that, by definition, for any $(w_1,\dots,w_k)\in {\cal A}$, it holds
$$\int_{\{x_n=t\}}f_i(x)d\Ham(x)=\int_{\{x_n=t\}}w_i(x)d\Ham(x),\qquad i=1,\dots,k,$$for all $t\in[-1,1]$. This shows that $Sr(g_i)=Sr(w_i)=Sr(f_i)$, thus in fact, we only have to prove that $g_i$ is almost everywhere equal to a rotationally symmetric function with respect to the line $\R e_n$, $i=1,\dots,k$.  
For $u\in\Sn\cap e_n^\perp$, let $T_u\in O(n,e_n)$ be the reflection with respect to the hyperplane $u^\perp$. Notice that if $(w_1,\dots,w_k)\in{\cal A}$, then the $k$-tuple $(M_u(w_1),\dots,M_u(w_k))$, also belongs to ${\cal A}$, where $M_u(w_i):=M(w_i;Id,T_u)$. Hence, if $\{(w_1^m,\dots,w_k^m)\}_{m=1}^\infty$ is a sequence from ${\cal A}$ that converges to $(g_1,\dots,g_k)$, then the sequence $\{(M_u(w_1^m),\dots,M_u(w_k^m))\}_{m=1}^\infty$ is also from ${\cal A}$ and converges to $(M_u(g_1),\dots,M_u(g_k))$. It follows that $(M_u(g_1),\dots,M_u(g_k))$ is also contained in ${\cal C}$. Using the trivial fact that for any $\varphi\in L^2(\Sn)$, it holds $\|\varphi\circ T_u\|_{L^2}=\|\varphi\|_{L^2}$, the fact that $Sr(f_i)=Sr(f_i)\circ T_u$ and the triangle inequality, we obtain
\begin{eqnarray*}
&&\|(M_u(g_1),\dots,M_u(g_k))-(Sr(f_1),\dots,Sr(f_k))\|\\
&\leq& \frac{1}{2}\|(g_1,\dots,g_k)-(Sr(f_1),\dots,Sr(f_k))\|+\frac{1}{2}\|(g_1\circ T_u,\dots,g_k\circ T_u)-(Sr(f_1)\circ T_u,\dots,Sr(f_k)\circ T_u)\|\\
&=&\frac{1}{2}d+\frac{1}{2}d=d.
\end{eqnarray*} 
It follows that $(M_u(g_1),\dots,M_u(g_k))=(g_1,\dots,g_k)$ (as elements of $X$), thus $g_i\circ T_u=g_i$ almost everywhere in $S^{n-1}$, for all $u\in \Sn\cap e_n^\perp$. This is enough to prove our claim.  
\end{proof}

\subsection{Reduction to surfaces of revolution} 
\hspace*{1.5em}Let $K$ be a convex body in $\Rn$ and $U$ be an open subset of $\Sn$. For technical reasons, we set
$f^{(j)}_{K,U}:=f^{(j)}_K\mathbbm{1}_U$, where $\mathbbm{1}_U$ is the indicator function of $U$ and $j\in\{1,\dots,n-1\}$.

\begin{lemma}\label{l-MK}
Let $K$ be a convex body in $\Rn$ and $U=\{x\in\Sn:x_n>a\}$, for some $0<a<1$. Assume that $S_1(K,\cdot)|_{\Bo(U)}$ is absolutely continuous and that for almost every direction $u$ in $U$, (\ref{eq:equal-princ-radii}) holds. Then, $Sr(h_K)$ is the support function of a convex body of revolution $MK$, which has the properties that $S_1(MK,\cdot)|_{\Bo(U)}$ is absolutely continuous and that for almost every direction $u$ in $U$, (\ref{eq:equal-princ-radii}) holds for $MK$ at $u$.
\end{lemma}
\begin{proof}
Without loss of generality we may assume that $K$ contains the origin in its interior. Therefore, there exist Euclidean balls $B_1,B_2$, centered at the origin, such that $B_1\subseteq K\subseteq B_2$. Moreover, by assumption and by Lemma \ref{l-1-Newton}, we have $f^{(1)}_{K,U}=\left(f^{(2)}_{K,U}\right)^{1/2}$, almost everywhere in $U$. Since $f^{(2)}_{K,U}\in L^1$, it follows that $f^{(1)}_{K,U}\in L^2$. Moreover, by Proposition \ref{prop-conv-sr}, for $k=2$, there exists a sequence $T_1^1,\dots,T_{m_1}^1,T_1^2,\dots,T^1_{m_2},\dots\in O(n,e_n)$, such that 
$$h_j:=M(h_K;T_1^j,\dots,T_{m_j}^j)\xrightarrow{j\to\infty}Sr(h_K)$$
and
$$M(f^{(1)}_{K,U};T_1^j,\dots,T_{m_j}^j)\xrightarrow{j\to\infty}Sr(f^{(1)}_{K,U})$$
in $L^2$ and (by taking subsequences) almost everywhere. Since $h_j=(1/m_j)(h_{(T_1^j)^\ast}+\dots+h_{(T_{m_j}^j)^\ast})$, $h_j$ is also a support function of some convex body $K_j$, where $B_1\subseteq K_j\subseteq B_2$, $j=1,2,\dots$. Thus, by the Blaschke Selection theorem, by taking a subsequence of $\{K_{_j}\}$ if necessary, we may assume that $\{K_{_j}\}$ converges to some convex body $\overline{MK}$ in the Hausdorff metric. Then, $h_{K_j}\to h_{\overline{MK}}$ (uniformly in $\Sn$), which shows that $h_{\overline{MK}}=h_{Sr(h_K)}$ and $\overline{MK}=MK$. Next, notice that 
$$f^{(1)}_{K_j,U}=\frac{f^{(1)}_{(T^j_1)^\ast K,U}+\dots+f^{(1)}_{(T^j_{m_j})^\ast K,U}}{m_j}=M(f^{(1)}_{K,U};T_1^j,\dots,T_{m_j}^j),$$ which converges in $L^2$ and thus weakly to $Sr(f^{(1)}_{K,U})$. This, in particular, shows that $S_1(MK,\cdot)|_{\Bo(U)}$ is absolutely continuous and that $f^{(1)}_{MK,U}=Sr(f^{(1)}_{K,U})$. Moreover, using Lemma \ref{l-2-Newton}, we see that $f^{(1)}_{K_j,U}=\left(f^{(2)}_{K_j,U}\right)^{1/2}$, almost everywhere in $U$, thus $f^{(2)}_{K_j,U}$ converges to $Sr(f^{(1)}_{K,U})^2$, almost everywhere in $U$. Let $\varphi:\Sn\to\R$ be any continuous non-negative function, supported inside $U$. Then, by Fatou's lemma and by the fact that $S_2(K_j,\cdot)$ converges weakly to $S_2(MK,\cdot)$, we get
\begin{eqnarray*}
\int_{\Sn}\left(f^{(1)}_{MK,U}\right)^2\varphi d\Ha&=&\int_{\Sn}\left(Sr(f^{(1)}_{K,U})\right)^2\varphi d\Ha\leq \liminf_{j\to\infty}\int_{\Sn}f^{(2)}_{K_j,U}\varphi d\Ha\\
&\leq & \liminf_{j\to\infty}\int_{\Sn}\varphi dS_2(K_j,\cdot)=\int_{\Sn}\varphi dS_2(MK,\cdot).
\end{eqnarray*}
Since $\varphi$ is arbitrary, we conclude by Lemma \ref{l-mt} $(i)$ that $\left(f^{(1)}_{MK,U}\right)^2d\Ha|_{\Bo(U)}\leq S_2(MK,\cdot)|_{\Bo(U)}$, which by Lemma \ref{l-mt} $(ii)$ gives $\left(f^{(1)}_{MK,U}\right)^2\leq f^{(2)}_{MK,U}$, almost everywhere in $U$. Thus, using Lemma \ref{l-1-Newton}, we see that 
for almost every direction $u$ in $U$, (\ref{eq:equal-princ-radii}) holds for $MK$ at $u$, concluding our proof.
\end{proof}
\begin{proposition}\label{prop-bodies-of-rev}
Let $K_1,\dots,K_{n-1}$ be convex bodies of revolution with respect to the axis $\R e_n$ and let $U=\{x\in \Sn:x_1>a\}$, for some $0<a<1$. For $i=1,\dots,n-1$, consider the Borel measure $\mu_i$ on the sphere, given by
\begin{equation*}\label{eq-mu_i-def}
\mu_i(\omega)=S_{n-1}(K_i,\omega\cap U)+S_{n-1}(K_i,(-\omega)\cap U).
\end{equation*} 
If none of the $K_1,\dots,K_{n-1}$ is a cylinder, then there are uniquely determined symmetric convex bodies $K_1^U,\dots,K_{n-1}^U$ of revolution with respect the the axis $\R e_n$, whose surface area measure equals $\mu_1,\dots,\mu_{n-1}$, respectively and
\begin{equation}\label{eq-mixed-area-bodies-of-rev}
S(K_1^U,\dots,K_{n-1}^U,\omega)=S(K_1,\dots,K_{n-1},\omega\cap U)+S(K_1,\dots,K_{n-1},(-\omega)\cap U), 
\end{equation}
for all $\omega\in\Bo(\Sn)$.
\end{proposition}
\begin{proof}
Let $i\in\{1,\dots,n-1\}$. Since $K_i$ is not a cylinder, it is clear that $\mu_i$ is not concentrated on any great subsphere of $\Sn$. Thus, by the Minkowski Existence and Uniqueness theorem, there exists a unique symmetric body of revolution (since $\mu_i$ is even and rotationally symmetric) $K_i^U$ with respect to the $x_n$-axis, whose surface area measure equals $\mu_i$. There is a simple geometric description of $K_i^U$: Since $U$ is contained in the hemisphere $\Sn\cap\{x_n>0\}$, there is a continuous, concave, non-increasing function $\varphi_i:[0,d_ie_{n-1}]\to\R$, for some $d_i>0$, such that the surface of revolution $\tau(K_i,U)$ is obtained by revolving the graph of $\varphi_i|_{[0,d_ie_{n-1})}$ about the $x_n$-axis. It follows easily by (\ref{eq:surface-area-measure}) that  $(\textnormal{bd}K_i^U)\cap \{x_n\geq 0\}$ is obtained by rotating the graph of the function $\widetilde{\varphi_i}:=\varphi_i-\varphi_i(d_i)$ about the $x_n$-axis. In the case that $K_1,\dots,K_{n-1}\in {\cal C}_+^2$, $S(K_1,\dots,K_{n-1},\cdot)|_{\Bo(U)}$ has density given by (\ref{eq:mixed-discriminant}) and since $h_{K_i}$ at any point in $U$ depends only on the function $\varphi_i$, $i=1,\dots,n-1$, it follows that $S(K_1^U,\dots,K_{n-1}^U,\cdot)|_{\Bo(U)}$ also has density; the same as the density of $S(K_1,\dots,K_{n-1},\cdot)|_{\Bo(U)}$. 
In the general case, one can approximate $K_1,\dots,K_{n-1}$ by sequences of ${\cal C}_+^2$ bodies of revolution. Since the corresponding sequence of mixed area measures converges weakly to $S(K_1,\dots,K_{n-1},\cdot)$, we conclude that for any continuous function $\phi:\Sn\to\R$, supported inside $U$, we have
$$\int_U\phi dS(K_1,\dots,K_{n-1},\cdot)=\int_U\phi dS(K_1^U,\dots,K_{n-1}^U,\cdot).$$
Hence, by Lemma \ref{l-mt} $(i)$, it follows that $S(K_1^U,\dots,K_{n-1}^U,\omega)=S(K_1,\dots,K_{n-1},\omega)$, for any $\omega\in\Bo(U)$.  The fact that (\ref{eq-mixed-area-bodies-of-rev}) holds for all $\omega\in \Bo(U\cup -U)$ follows trivially by symmetry.

It remains to prove that $S(K_1^U,\dots,K_{n-1}^U,\Sn\setminus (U\cup-U))=0$. Notice that for any $u\in (\Sn\setminus U)\cap\textnormal{span}\{e_{n-1},e_n\}\cap\{x_n\geq 0\}$, the intersection of the supporting line to the graph of $\widetilde{\varphi_i}$, whose outer unit normal vector is $u$, with  the graph of $\widetilde{\varphi_i}$, contains only the point $d_ie_{n-1}$, $i=1,\dots,n-1$. Hence, by the rotational symmetry and central symmetry of $K_i^U$, we conclude that for any $u\in \Sn\setminus(U\cup-U)$, it holds $F(K_i^U,u)\subseteq d_iS^{n-1}\cap e_n^\perp$, $i=1,\dots,n-1$. The additivity of the support set functional (\ref{eq:support-set}) gives $F(K_1^U+\dots+K_{n-1}^U,u)\subseteq (d_1+\dots+d_{n-1})S^{n-1}\cap e_n^\perp$.
In other words, $\tau(K_1^U+\dots K_{n-1}^U,\Sn\setminus (U\cup-U))=(d_1+\dots+d_{n-1})S^{n-1}\cap e_n^\perp$, which by (\ref{eq:surface-area-measure}) gives $S_{n-1}(K_1^U+\dots+K_{n-1}^U,\Sn\setminus(U\cup-U))=0$. It follows immediately by (\ref{eq:mixedarea-additivity}) that $S(K_1^U,\dots,K_{n-1}^U,\Sn\setminus (U\cup-U))=0$, as asserted.
\end{proof}
\subsection{Regularity}
\begin{lemma}\label{l-umb-bodies-of-revolution}
Let $K$ be a convex body in $\Rn$ and $U$ be a spherical cap, centered in $e_n$. If $K$ and $U$ satisfy the assumptions of Theorem \ref{thm-main}, then $Sr(f^{(1)}_{K,U})$ equals to a constant, almost everywhere in $U$.
\end{lemma}
\begin{proof} Recall that by Lemma \ref{l-MK}, it holds $Sr(f_{K,U}^{(1)})=f^{(1)}_{MK,U}=\left(f^{(2)}_{MK,U}\right)^{1/2}$, almost everywhere in $U$. Also, by Proposition \ref{prop-bodies-of-rev}, (\ref{eq:mixarea}) and the Alesandrov-Fenchel inequality (\ref{eq:af}), we have
\begin{eqnarray*}
\frac{1}{n}\int_Uf^{(1)}_{MK}d\Ha&=&\frac{1}{n}\int_UdS(MK,B_2^n[n-2],\cdot)=\frac{1}{n}\int_UdS((MK)^U,(B_2^n)^U[n-2],\cdot)\\
&=&\frac{1}{2}V(B_2^n,(MK)^U,(B_2^n)^U[n-2])\\
&\geq& \frac{1}{2}\left(V(B_2^n,(MK)^U[2],(B_2^n)^U[n-3])V(B_2^n,(B_2^n)^U[n-1])\right)^{1/2}\\
&=& \left(\frac{1}{n}\int_UdS((MK)^U[2],(B_2^n)^U[n-3],\cdot)\frac{1}{n}\int_UdS((B_2^n)^U[n-1],\cdot)\right)^{1/2}\\
&\geq& \frac{1}{n}\left(\Ha(U)\right)^{1/2}\left(\int_Uf^{(2)}_{MK}d\Ha\right)^{1/2}\\
&=&\frac{1}{n}\left(\Ha(U)\right)^{1/2}\left(\int_U\left(f^{(1)}_{MK}\right)^2d\Ha\right)^{1/2}
\end{eqnarray*}
On the other hand, the Cauchy-Schwartz inequality gives
\begin{equation}\label{eq-CS-last}
\int_Uf_{MK}^{(1)}d\Ha\leq \left(\Ha(U)\right)^{1/2}\left(\int_U\left(f^{(1)}_{MK}\right)^2d\Ha\right)^{1/2}.
\end{equation}
Therefore, there must be equality in the Cauchy-Schwartz inequality (\ref{eq-CS-last}), which is only possible if $f^{(1)}_{MK}$ is equal to a constant almost everywhere in $U$, proving our claim.
\end{proof}
\textit{}\\
Proof of Theorem \ref{thm-main}.\\
Let $K$, $U$ be as in the statement of Theorem \ref{thm-main}.
Without loss of generality, we may assume that $U$ is a spherical cap centered at $e_n$. 
By Lemma \ref{l-umb-bodies-of-revolution}, $Sr(f_{K,U}^{(1)})=f^{(1)}_{MK,U}$ can be taken to be equal to a constant $c\geq 0$ in $U$. For $u\in U$, define the following quantity (if it exists)
$$F(u):=\lim_{\textnormal{diam}(U')\to 0}\frac{\int_{U'}f^{(1)}_{K}d\Ha}{\Ha(U')},$$
where $U'$ runs over all spherical caps $U'\subseteq U$, whose center is $u$. 
First assume that $u=e_n$ and let $U'\subseteq U$ be a spherical cap centered at $e_n$. Notice, also, that $Sr(f^{(1)}_{K,U'})|_{U'}=Sr(f^{(1)}_{K,U})|_{U'}$. Then, by Lemma \ref{l-sr-l2}, it follows that $\int_{U'}f^{(1)}_{K}d\Ha=\int_{U'}Sr(f^{(1)}_{K,U})d\Ha=c\Ha(U')$. In particular, $F(e_n)$ exists and equals to $c$. Moreover, notice that if $e_n$ is a Lebesgue point of $f^{(1)}_{K}$, then
$F(e_n)=f^{(1)}_{K}(e_n)$. Next, take any spherical cap $V$ inside $U$, centered at some $v\in U$. Since the pair $(K,V)$ also satisfies the assumptions of Theorem \ref{thm-main} and since $e_n$ can clearly be replaced by any other point on the sphere, our previous discussion shows that $F(v)$ exists and
\begin{equation}\label{eq-mvp}
\frac{\int_Vf^{(1)}_{K}d\Ha}{\Ha(V)}=F(v),
\end{equation}
while $F(v)$ equals $f^{(1)}_{K}(v)$ if $v$ is a Lebesgue point of $f^{(1)}_{K}$. In particular, the function $F:U\to\R$ is well defined in $U$. Notice, however, that since almost every $v\in U$ is a Lebesgue point of $f^{(1)}_{K}$, $F$ equals $f^{(1)}_{K}$ almost everywhere in $U$. Thus, by (\ref{eq-mvp}), it holds
$$\frac{\int_VFd\Ha}{\Ha(V)}=F(v),$$
for all $v\in U$ and for all spherical caps $V\subseteq U$, centered at $v$. Thus, $F$ has the so-called mean value property, which on the sphere (just like in the Euclidean case) implies that $F$ is harmonic \cite{Wil}. It follows using e.g. \cite[Proposition 1.6]{Ta}, that $F$ is $C^\infty$-smooth (actually real analytic). Consequently, $f^{(1)}_{K}$ is almost everywhere equal to a $C^\infty$ function in $U$. Since (\ref{eq:distributions}) holds in the sense of distributions in $U$, it follows again by \cite[Proposition 1.6]{Ta} that $h_K$ is of class $C^\infty$ on $U$. Next, notice that, by Lemma \ref{l-2-Newton}, the pair $(K+B_2^n,U)$ also satisfies the assumptions of Theorem \ref{thm-main}. Since $f^{(1)}_{K+B_2^n}\geq 1>0$ in $U$, it follows that all principal radii of curvature of $K+B_2^n$ are strictly positive, thus (since $h_{K+B_2^n}$ is smooth) as in \cite[pp 120]{Sc} we conclude that $\tau(K+B_2^n,U)$ is smooth as a manifold. This, together with (\ref{eq:rad-curv}) and Theorem \ref{thm-old}, shows that $\tau(K+B_2^n)$ is contained in a Euclidean sphere. Therefore, and since $\tau(K+B_2^n,U)$ is open in $\textnormal{bd}K$, we conclude that $h_{K+B_2^n}+\langle a,\cdot\rangle $ is constant on $U$, for some fixed vector $a$, and hence $h_K+\langle a,\cdot\rangle $ is constant on $U$, ending the proof of Theorem \ref{thm-main}. $\square$
\section{Even functions with isotropic sections}

A {\it zonoid} $Z$ is a convex body whose support function is the cosine transform of some (positive) Borel measure $\mu$ on $\Sn$. The measure $\mu$ is called the {\it generating measure} of $Z$.
 
Let $Z_1,\dots,Z_{n-1}$ be zonoids in $\Rn$ with corresponding generating measures $\mu_1,\dots,\mu_{n-1}$. If $\mu_1,\dots, \mu_{n-1}$ are absolutely continuous with corresponding densities $g_1,\dots,g_{n-1}$, then there is an integral-geometric formula, essentially due to W. Weil \cite{W} (see also  \cite[Section 5.3]{Sc}) that gives the density of the mixed area measure $S(Z_1,\dots,Z_{n-1},\cdot)$.
\begin{flalign}
& \  \ \  \ \frac{dS(Z_1,\dots,Z_{n-1},\cdot)}{d\Ha(\cdot)}(u)&\nonumber
\end{flalign}
\begin{equation}\label{eq-Weil-1}
=\frac{2^{n-1}}{(n-1)!}\int_{\Sn\cap u^\perp}\dots\int_{\Sn\cap u^\perp}\textnormal{det}(x_1,\dots,x_{n-1})^2g_1(x_1)\dots g_{n-1}(x_{n-1})d\Ham(x_1)\dots d\Ham(x_{n-1}).
\end{equation}  
In the particular case that $Z_1=\dots=Z_k=Z$, $g_1=\dots=g_k=g$, $Z_{k+1}=\dots=Z_{n-1}=B_2^n$, $k=1,\dots, n-1$, we have
$$h_{Z_i}(u)=a_n\int_{\Sn}|\langle x,u\rangle|d\Ha(x),\qquad\textnormal{where}\qquad a_n=\left(\int_{\Sn}|x_1|d\Ha(x)\right)^{-1},$$
$i=j+1,\dots,n-1$. Hence, (\ref{eq-Weil-1}) becomes $ f^{(j)}_Z(u)=$
\begin{equation}\label{eq-Weil-2}
\frac{a_n^{n-j-1}2^{n-1}}{(n-1)!}\int_{\Sn\cap u^\perp}\dots\int_{\Sn\cap u^\perp}\textnormal{det}(x_1,\dots,x_{n-1})^2g(x_1)\dots g(x_{j})d\Ham(x_1)\dots d\Ham(x_{n-1}).
\end{equation}  
In particular, area measures of any order of the zonoid $Z$ are absolutely continuous, if the generating measure of $Z$ is absolutely continuous. Notice, also that (\ref{eq-Weil-2}) implies that
\begin{equation}\label{eq-Weil-3}
f^{(1)}_Z(u)=b_n\int_{\Sn\cap u^\perp}g(x)d\Ham(x)=b_n{\cal R}(g),
\end{equation}
where $b_n>0$ is a constant that depends only on the dimension. 
\begin{lemma}\label{l-last-1}
Let $n\geq 3$, $U$ be an open set in  $\Sn$ and $g:\Sn\to\R_+$ be a bounded measurable function. Denote by $Z(g)$ the zonoid with generating measure $gd\Ha(\cdot)$.
The following are equivalent.
\begin{enumerate}[i)] 
\item The restriction $g|_{\Sn\cap u^\perp}$ is isotropic for almost every $u\in U$.
\item For almost every $u\in U$, it holds 
\begin{equation}\label{eq-last-s1-s2}
\left(f^{(1)}_{Z(g)}(u)\right)^2=f^{(2)}_{Z(g)}(u).
\end{equation}
\end{enumerate} 
\end{lemma}
\begin{proof}
Assume that $(i)$ holds. For any $u\in\Sn$, for which $g|_{\Sn\cap u^\perp}$, it holds (just expand the determinant and use the fact that $\int_{\Sn\cap u^\perp}\langle x,e_i\rangle\langle x,e_j\rangle d\Ham(x)=0$, for $i\neq j$)
\begin{eqnarray*}
&&\int_{\Sn\cap u^\perp}\dots\int_{\Sn\cap u^\perp}\textnormal{det}(x_1,\dots,x_{n-1})^2g(x_1)g(x_2)d\Ham(x_1)\dots d\Ham(x_{n-1})\\
&=&c_n\left(\int_{\Sn\cap u^\perp}g(x)d\Ham(x)\right)^2,
\end{eqnarray*}
where $c_n$ is a positive constant  that depends only on the dimension $n$. Combining with (\ref{eq-Weil-2}), (\ref{eq-Weil-3}) and the assumption, we arrive at
$$\left(f^{(1)}_{Z(g)}(u)\right)^2=d_nf^{(2)}_{Z(g)}(u),$$
for almost every $u\in U$, where $d_n>0$ again depends only on $n$. However, if $g\equiv a_n$ on $\Sn$, that is $Z(g)=B_2^n$, we already know that (\ref{eq-last-s1-s2}) holds, thus $d_n=1$. This proves $(ii)$. The proof that $(ii)$ implies $(i)$ is similar and we omit it.
\end{proof}
\textit{}\\
Proof of Theorem \ref{thm-counterexample}.\\
We know (see \cite[Theorem 3.5.4]{Sc}) that if $G:\Sn\to\R$ is an even smooth enough function, then there exists an even continuous function $w:\Sn\to\R$, such that
\begin{equation}\label{eq-w(x)}
G={\cal C}(w).
\end{equation}
Let $V$ be a spherical cap which is disoint from $U$ and $G$ be an even function of class $C^\infty$, such that $G|_U\equiv 1$ and $G|_V\equiv 2$ and let $w$ be the corresponding function in the integral representation (\ref{eq-w(x)}). Set $g:=w+1+\max_{u\in S_n}|w(u)|$. Then, $g>0$ on $\Sn$ and ${\cal C}(g)$ is the support function of the zonoid $Z(g)$, which is constant on the open sets $U$ and $V$. Hence, $\tau(Z(g),U)$ and $\tau(Z(g),V)$ are contained in spheres of radii 1 and 2 respectively. Thus, $f^{(1)}_{Z(g)}(u)=1$, for all $u\in U$ and $f^{(1)}_{Z(g)}(v)=2$, for all $v\in V$. On the other hand, $\left(f^{(1)}_{Z(g)}(u)\right)^2=f^{(2)}_{Z(g)}(u)$, for all $u\in U$. This, together with Lemma \ref{l-last-1}, shows that $g|_{\Sn\cap U^\perp}$ is isotropic for all $u\in U$. However, since $U^\perp\cap V^\perp\neq \emptyset$, (\ref{eq-Weil-3}) shows that $g$ cannot be constant on $U^\perp$ (or in $V^\perp$). $\square$\\
\\
Proof of Theorem \ref{thm-meta-main-1}.\\
Let us first extend $g$ to the whole $\Sn$, so that $f|_{\Sn\setminus U}\equiv 0$. Since for any two spherical caps $V_1,V_2\subseteq S^{n-1}$, it holds $V_1^\perp\cap V_2^\perp\neq \emptyset$, we may assume that $U$ is a spherical cap. Notice that if $g$ satisfies the assumptions of Theorem \ref{thm-meta-main-1}, then $g+c$ also satisfies the assumptions of Theorem \ref{thm-meta-main-1}, so since $g$ is bounded, we may assume $g$ to be non-negative. Denote, again, by $Z(g)$ the zonoid with generating measure $gd\Ha(\cdot)$. Lemma \ref{l-last-1} and the assumption show that 
$$\left(f^{(1)}_{Z(g)}\right)^2=f^{(2)}_{Z(g)} \ ,$$ almost everywhere in $U$. Since $S_1(Z(g),\cdot)$ is absolutely continuous, it follows by Theorem \ref{thm-main} that $\tau(Z(g),U)$ is contained in a sphere. In particular,
${\cal C}(g)|_U=h_{Z(g)}|_U=c+\langle a,\cdot\rangle$ and $b_n{\cal R}(g)=f^{(1)}_{Z(g)}=c'$, for some constants $c,c'>0$ and for some vector $a\in\Rn$. $\square$
\\
\\
Before ending this note, we would like to state, under some regularity assumptions on $g$, a local version of Theorem \ref{thm-meta-main-1}.  
\begin{theorem}\label{thm-meta-main-2}
Let $n\geq 4$ and $g:\Sn\to\R$ be a smooth enough function, so that the cosine transform of the measure $gd\Ha(\cdot)$ is of class $C^2$. Assume, furthermore, that there exist $k\geq 3$, $H\in G_{n,k}$ and an open set $U$ in $H$, such that $g|_{\Sn\cap u^\perp}$ is isotropic, for all $u\in U$. Then, $({\cal R}g)|_U$ is constant.
\begin{proof}
Again, we may assume that $g>0$. Then, $Z(g)$ is of class ${\cal C}^2_+$ (the same holds of course for $Z(g)|H$) and therefore it is meaningful to consider (\ref{eq:equal-princ-radii}) for $Z(g)$ pointwise. Let $u\in U$. As in Lemma \ref{l-last-1}, we see that (\ref{eq:equal-princ-radii}) holds for $Z(g)$ at $u$. Let $\{\varepsilon_1,\dots,\varepsilon_{k-1}\}$ be an orthonormal basis of $H\cap u^\perp$ and extend it to an orthonormal basis $\{\varepsilon_1,\dots,\varepsilon_{n-1}\}$ of $u^\perp$. It holds
$$Hess (h_Z(g))(u)_{(n-1)\times(n-1)}=r(u)I_{(n-1)\times(n-1)},$$
where the differentiation is with respect to the basis $\{\varepsilon_1,\dots,\varepsilon_{n-1}\}$ (or any orthonormal basis in $u^\perp$) and $r(u)>0$ is the common value of the principal radii of $\textnormal{bd}Z(g)$ at $u$.  
This shows that $Hess (h_{(Z|H)(g)})(u)_{(k-1)\times(k-1)}$ is also $r(u)$ times the $(k-1)\times(k-1)$ identity matrix, when the differentiation is with respect to the basis $\{\varepsilon_1,\dots,\varepsilon_{k-1}\}$. 
Consequently, for any $u\in U$, (\ref{eq:equal-princ-radii}) holds for $Z(g)|H$ at $u$.  Using Theorem \ref{thm-main}, we conclude that $\tau(Z(g)|H,U)$ is contained in a $k$-dimensional sphere, thus $r(u)$ is constant in $U$.
Finally, as in the proof of Lemma \ref{l-last-1}, one can easily see that 
$$r(u)=\frac{1}{n-1}\int_{\Sn\cap u^\perp}g(x)d\Ham(x),$$ 
which by Theorem \ref{thm-old} completes our proof.
\end{proof}
\end{theorem}
\textit{}\\
{\bf Acknowledgement.} We are grateful to Daniel Hug for his help and interest in this work and for discovering errors in the statement and proof of previous version of Theorem \ref{thm-main}. In particular, Example \ref{ex-1} is due to him. We would also like to thank Dmitry Ryabogin for some excellent discussions concerning problems related to Problem \ref{prob1} and Andreas Savas-Halilaj for providing us references \cite{Doc} and \cite{S-V} and for related discussions.

\vspace{0.8 cm}
 
\noindent Ioannins Purnaras \\
Department of Mathematics\\
University of Ioannina\\
Ioannina, Greece, 45110 \\
E-mail address: ipurnara@uoi.gr 

\vspace{0.5 cm}

\noindent Christos Saroglou \\
Department of Mathematics\\
University of Ioannina\\
Ioannina, Greece, 45110 \\
E-mail address: csaroglou@uoi.gr \ \&\ christos.saroglou@gmail.com

\end{document}